\documentclass[12pt,a4paper,reqno]{amsart}
\usepackage{amsfonts}
\usepackage{amsthm}
\usepackage{amsmath}
\usepackage{amscd}
\usepackage[latin2]{inputenc}
\usepackage{t1enc}
\usepackage[mathscr]{eucal}
\usepackage{indentfirst}
\usepackage{graphicx}
\usepackage{graphics}
\usepackage{pict2e}
\usepackage{epic}
\numberwithin{equation}{section}
\usepackage[margin=2.8cm]{geometry}
\usepackage{hyperref}
\hypersetup{
    colorlinks=true,
    linkcolor=blue,
    filecolor=blue,      
    urlcolor=cyan,
}

\usepackage{epstopdf} 
\usepackage{tabularx}
\usepackage{longtable}
\usepackage{multirow}
\usepackage{hhline}
\usepackage{fancyhdr}
\usepackage{stmaryrd}
\usepackage[all]{xy}
\usepackage{xcolor}
\usepackage{comment}
\usepackage{shuffle}
\usepackage{nameref, cleveref}
\usepackage{mathtools}
\usepackage{enumerate}
\usepackage{tikz-cd}

\theoremstyle{plain}
\newtheorem{Th}{Theorem}[section]
\newtheorem{Lemma}[Th]{Lemma}
\newtheorem{Cor}[Th]{Corollary}
\newtheorem{Prop}[Th]{Proposition}

\theoremstyle{definition}
\newtheorem{Def}[Th]{Definition}
\newtheorem{Conj}[Th]{Conjecture}

\newtheorem{?}[Th]{Problem}
\newtheorem{Ex}[Th]{Example}

\begin{document}

\title[Zero Distribution of $v-$adic MZVs]{Zero Distribution of $v-$adic Multiple Zeta Values over $\mathbb{F}_q(t)$}

\author[Qibin Shen]{Qibin Shen}

\address{Department of Mathematics\\University of Rochester \\ 500 Joseph C. Wilson Blvd., Rochester, NY, 14627} 

\email{qshen4@ur.rochester.edu}

\date{September 20, 2019}

\subjclass[2010]{Primary: 11M32. Secondary: 11R58.}

\keywords{$v-$adic Multiple Zeta Values}

\begin{abstract} This paper aims to study the zero distribution of $v-$adic multiple zeta values over function fields. We show that the interpolated $v-$adic MZVs at negative integers only vanish at what we call the ``trivial zeros'', for degree one prime over rational function fields. And we conjecture that this result can be generalized to all primes.
\end{abstract}

\maketitle

\section{Introduction} 

Multiple Zeta Values(MZVs) was originally introduced and studied by Euler, and recently, these values showed up
again in various subjects in mathematics and
mathematical physics. In Furusho's paper \cite{F04}, by making an analytic continuation of $p-$adic multiple polylogarithms introduced by Coleman's $p-$adic iterated integration theory \cite{Col} and he was able to define $p-$adic multiple
zeta values to be a limit value at $1$ of analytically continued $p-$adic multiple
polylogarithms.

It's well known that researches on function fields and number fields often go in parallel and thus mutually beneficial. Often the former one is inspired by the latter.

Unlike the classical case, the interpolated $v-$adic MZVs in function field case are well-defined at all integer points. In this paper, we mainly studied $v-$adic multiple zeta values over $\mathbb{F}_q(t)$.

We first adapt some notations before we give the definition.

\textbf{Notations:}
\begin{eqnarray*}
\mathbb{Z}_- &=& \{ \textit{negative integers} \},\\
q &=& p^f, \textit{a power of a prime p},\\
\mathbb{F}_q &=& \textit{a finite field of $q$ elements},\\
K &=& \textit{function field over $\mathbb{F}_q$},\\
\infty &=& \textit{a rational place in $K$},\\
A &=& \textit{the ring of integral elements outside $\infty$},\\
v &=& \textit{a monic prime in $A$},\\
A_{d+} &=& \textit{monics in $A$ of degree $d$},\\
K_\infty &=& \textit{completion of $K$ at $\infty$},\\
K_v &=& \textit{completion of $K$ at prime $v$}.
\end{eqnarray*}

\begin{Def}
An integer $s$ is called \textit{$q-$even} if $(q-1) | s$. Otherwise, it's called \textit{$q-$odd}.
\end{Def}
The reason to introduce ``$q-$even'' is that the behavior of Carlitz zeta values ($\zeta(s)$ of depth 1 defined in \Cref{mzv}) at $q-$even integers is similar as that of Riemann zeta values at even integers.

Now we are ready to define MZVs over $K$. Given $k \in \mathbb{Z}$ and $\mathbf{s}$ with $s_j \in \mathbb{Z}$, for $d \geq 0$, let
$$S_d(k):=\sum_{a \in A_{d+}} \frac{1}{a^k} \in K,$$
$$S_d(\mathbf{s}):=S_d(s_1)\sum_{d>d_2>\cdots>d_r \geq 0} S_{d_2}(s_2) \cdots S_{d_r}(s_r) \in K.$$
\begin{Def}\label{mzv}
Let $\mathbf{s}=(s_1, \ldots, s_r), s_i \in \mathbb{Z}$, define \textit{multiple zeta value} 
$$\zeta(\mathbf{s}):=\sum_{d_1 > \cdots > d_r \geq 0} S_{d_1}(s_1) \cdots S_{d_r}(s_r) \in K_\infty.$$
We call $r$ the depth of $\zeta(\mathbf{s})$, and if each $s_i>0$, $wt:=\sum_{i=1}^rs_i$ is the weight of $\zeta(\mathbf{s})$.
\end{Def}
We first introduce the interpolated $v-$adic multiple zeta values defined by Thakur in \cite{T04} here.
Given $k \in \mathbb{Z}$, $v$ a prime in $A_+$ and $\mathbf{s}$ with $s_i \in \mathbb{Z}$, for $d \geq 0$, let
$$S_d(k):=\sum_{a \in A_{d+}} \frac{1}{a^k} \in K,$$
$$\widetilde{S}_d(k):=\sum_{\substack{a \in A_{d+}\\ (a,v)=1}} \frac{1}{a^k} \in K,$$
$$\widetilde{S}_d(\mathbf{s}):=\widetilde{S}_d(s_1)\sum_{d>d_2>\cdots>d_r \geq 0} \widetilde{S}_{d_2}(s_2) \cdots \widetilde{S}_{d_r}(s_r) \in K.$$
\begin{Def}
Let $\mathbf{s}=(s_1, \ldots, s_r), s_i \in \mathbb{Z}$, $v$ a prime in $A_+$, define \textit{$v-$adic multiple zeta value} 
$$\zeta_v(\mathbf{s}):=\sum_{d_1 > \cdots > d_r \geq 0} \widetilde{S}_{d_1}(s_1) \cdots \widetilde{S}_{d_r}(s_r) \in K_v.$$
\end{Def}
It's well known that this power series is convergent for any $\mathbf{s}$ with each $s_i\in \mathbb{Z}$. These values can be extended continuously to $v-$adic domains, but we will be only interested in special values at $\mathbf{s}$ with integer coordinates.

\begin{Def}
Recall $q=p^f$.
\begin{itemize}
\item[(1)] For $k \in \mathbb{N}$, let $l(k)$ be the sum of digits of $k$ base $q$. Define
$$L_k:= \min_{i=0,\ldots,f-1} \{ l(kp^i)/(q-1) \}.$$
Note that $k \equiv l(k)\  (\mathrm{mod}\ q-1)$, thus $L_k$ is an integer if and only if $(q-1)|k$, i.e. $k$ is $q-$even.
\item[(2)] For $d \geq 0$ and $k > 0$, define
$$V_d(k):=\{ (m_0, \ldots, m_d) \in U_d(k)\; \big|\;m_0 >0 \}.$$
\end{itemize}
\end{Def}

\begin{Def}\label{trivialzero}
Let $s_j \in \mathbb{Z}_-$. $\zeta (s_1, \ldots, s_r)=0$ trivially if there exists some $1 \leq i \leq r-1$ such that $r-i > L_{-s_i}$. We call such zeros \textit{trivial zeros}. Other zeros are called \textit{nontrivial}.
\end{Def}

In the works of Goss\cite{Goss1}, Dinesh\cite{T091}, and Shuhui\cite{Shuhui}, we have 
\begin{Th}
Let $K=\mathbb{F}_q(t)$, $\mathbf{s}=(s_1, \cdots, s_r)\in \mathbb{Z}^r$ and each $s_i$ shares the same sign, $\zeta(\mathbf{s})=0$ if and only if one of the following conditions holds
\begin{itemize}
    \item[(1)] $r>1$, $\mathbf{s}$ is a trivial zero,
    \item[(2)] $r=1$, $s$ is $q-$even.
\end{itemize}
\end{Th}

\section{Trivial Zeros for $v-$adic MZVs}
To study $\zeta_v(\mathbf{s})$ at negative integers, we first study the behaviour of $S_d(s)$ for $s < 0$.
\begin{align}
\nonumber
S_d(s) &= \sum_{\theta_i \in \mathbb{F}_q} (t^d+\theta_{1} t^{d-1}+ \cdots +\theta_d)^{-s} \\\nonumber
&= \sum_{\theta_i \in \mathbb{F}_q}\sum_{\substack{m_0+\cdots +m_d = -s\\ m_i \geq 0}} \binom{-s}{m_0, \ldots, m_d} \theta_1^{m_1} \cdots \theta_d^{m_d} t^{dm_0+(d-1)m_1+\cdots +m_{d-1}}\\\nonumber
&= (-1)^d\sum_{\substack{m_0+\cdots +m_d = -s\\ m_0 \geq 0, (q-1) | m_i > 0, 1 \leq i \leq d}} \binom{-s}{m_0, \ldots, m_d} t^{dm_0+(d-1)m_1+\cdots +m_{d-1}}\\
&= (-1)^d\sum_{\substack{\oplus_{0}^d m_i= -s\\ m_0 \geq 0, (q-1) | m_i > 0, 1 \leq i \leq d}} \binom{-s}{m_0, \ldots, m_d} t^{dm_0+(d-1)m_1+\cdots +m_{d-1}} \label{Sd(n)}
\end{align}
where $\oplus_{0}^d m_i$ denotes sum $\sum_{i=0}^d m_i$ with no carry over of digits base $p$. The third equality comes from exchanging two sum indices and the fact that $\sum_{\theta \in \mathbb{F}_q} \theta^k=-1$ if $k \neq 0, (q-1) | k$ and 0 otherwise. The last equality is the application of Lucas' theorem saying that the multinomial coefficient $\binom{-s}{m_1, \ldots, m_d}$ vanishes in $\mathbb{F}_q$ iff the sum $\sum m_i$ has carry over base $p$. Obviously, $S_d(s)$ vanishes if sum (\ref{Sd(n)}) is empty. 

In \cite{Car48}, Carlitz claimed the converse also holds. More precisely, he asserted that, if sum (\ref{Sd(n)}) is not empty, the term $t^{dm_0+(d-1)m_1+\cdots +m_{d-1}}$ with $(m_0, \ldots, m_d)$ lexicographically largest among the sum indices attains the unique maximal degree. Such $(m_0, \ldots, m_d)$ is called \textit{greedy}. This was proved by Diaz-Vargas in \cite{DV} for $q=p$ and Sheats in \cite{Sheats} for general $q$.
\begin{Th}[Calitz \& Sheats]\label{greedy}
For $k>0, d \geq 0$, denote by $U_d(k)$ the collection of all $(d+1)$-tuples $(m_0, \ldots, m_d)$ such that (1) $k=m_0 \oplus \cdots \oplus m_d$, and (2) for $1 \leq i \leq d$, $m_i >0$ and $(q-1)| m_i$. Then, for $s< 0$, $S_d(s) \neq 0$ if and only if $U_d(-s) \neq \emptyset$. Moreover, the term corresponding to the greedy element achieves the unique maximal degree.
\end{Th}

\begin{Prop}\label{vanishcondition}
$S_d(s) = 0 \Leftrightarrow d > L_{-s}$.
\end{Prop}
The proposition follows easily from the following lemma proved by Sheats and it was already pointed out by Boeckle. The notations and expression of the lemma are different from those in Sheats' paper, but it's easy to see the statements are equivalent.
\begin{Lemma}\textnormal{(\cite[Prop. 4.3(a)]{Sheats})}
$V_d(k) = \emptyset \Leftrightarrow d > L_k.$
\end{Lemma}
\begin{proof}[Proof of \Cref{vanishcondition}:]
By \Cref{greedy}, we need to show $U_d(k) = \emptyset$ iff $d > L_k$. There are two cases: $k$ is $q-$even or $q-$odd. We consider each separately. 

If $k$ is $q-$even, $U_d(k)=V_d(k) \cup \{ (0, m_1, \ldots, m_d)\ \big|\ (m_1, \ldots, m_d) \in V_{d-1}(k) \}$. $U_d(k)=\emptyset$ iff $V_d(k)=V_{d-1}(k)=\emptyset$, i.e. $d-1 \geq L_k$. Since $L_k$ is an integer, $d-1 \geq L_k \Leftrightarrow d > L_k$.

If $k$ is $q-$odd, $U_d(k)=V_d(k)$, thus $U_d(k)=\emptyset$ iff $d \geq L_k$.  As $L_k$ is not an integer in this case, $d \geq L_k \Leftrightarrow d > L_k$.
\end{proof}

Applying all the above results, we can show
\begin{Prop}\label{vanishcondition2}
$\widetilde{S}_d(s) = 0$ $\Leftrightarrow$ $d > L_{-s}+deg(v)$ or $deg(v)> d > \lfloor L_{-s}\rfloor$.
\end{Prop}
\begin{proof}
($\Leftarrow$) This direction follows easily from the above \Cref{vanishcondition} and the definition of $\widetilde{S}_d(s)$.\\
($\Rightarrow$) When $d<deg(v)$, $\widetilde{S}_d(s)=S_d(s)$, and hence this power sum is zero iff $d > \lfloor L_{-s}\rfloor$ by \Cref{vanishcondition}. When $deg(v)\leq d\leq L_{-s}+deg(v)$, we have $\widetilde{S}_d(s)=S_d(s)-v^{-s}S_{d-deg(v)}(s)$.

If $v=t$, we have $\widetilde{S}_d(s)=S_d(s)-t^{-s}S_{d-1}(s)$. Consider the $(t+1)-$adic valuation of $S_d(s)$ and $v^{-s}S_{d-deg(v)}(s)$ respectively.
Since $t^{-s}S_{d-1}$ has same $(t+1)-$adic valuation as $S_{d-1}$, which equals $\nu_{d-1}(s)$, so if $d>1$, by above proposition, we have $\nu_{d}(s)>\nu_{d-1}(s)$, hence $\widetilde{S}_d(s)=S_d(s)-t^{-s}S_{d-1}(s)\neq 0$.

If $v\neq t$, similarly, we have $\widetilde{S}_d(s)=S_d(s)-v^{-s}S_{d-1}(s)$. Consider the $t-$adic valuation of $S_d(s)$ and $v^{-s}S_{d-deg(v)}(s)$ respectively.
Since $v^{-s}S_{d-1}$ has same $t-$adic valuation as $S_{d-1}$, which equals $\nu_{d-1}(s)$, so if $d>1$, by above proposition, we have $\nu_{d}(s)>\nu_{d-1}(s)$, hence $\widetilde{S}_d(s)=S_d(s)-t^{-s}S_{d-1}(s)\neq 0$.

It remains to prove the case when $deg(v)=d=1$. Without loss of generality, we may assume $v=t$, then we have $\widetilde{S}_1(s)=\sum_{\theta\in \mathbb{F}_q^*}(t+\theta)^{-s}\neq 0$ since, for example, its Laurent series is  $(q-1)t^{-s}$ plus terms of higher valuations at infinity. 
\end{proof}

\begin{Cor}
Given a rational function field $K=\mathbb{F}_q(t)$, $(s_i)_{i=1}^r \in \mathbb{Z}^r$, $r>1$.

$\zeta_v(s_1, \ldots, s_r)=0$ if either one of the following conditions holds
\begin{itemize}
    \item[(1)] $\exists1 \leq i \leq r$ such that $s_i<0$, $r-i > L_{-s_i} + deg(v)$.
    \item[(2)] $\exists1 \leq i,j\leq r$ such that $s_i,s_j<0$, $deg(v)>r-i > L_{-s_i}$, and $i-j>L_{-s_j}$.
\end{itemize}
\end{Cor}
\begin{proof}
If $\exists i$ such that $r-i>L_{-s_i}+deg(v)$, then since the least relevant $d_i\geq r-i$, we apply above proposition and get $\widetilde{S}_{d_{i}}(s_i)=0$. Hence $\zeta_v(\mathbf{s})=0$.

If $\exists i,j$ such that $i>j$, $deg(v)>r-i > L_{-s_i}$, and $i-j>L_{s_j}$, then $\widetilde{S}_{d_{j}}(s_j)\widetilde{S}_{d_{i}}(s_i)=0$. This is because that by above proposition, if $\widetilde{S}_{d_{i}}(s_i)\neq 0$, i.e., $d_i\geq deg(v)$, we must have $d_j\geq deg(v)-(r-i)+(r-j)=deg(v)+i-j>deg(v)+L_{-s_j}$, i.e., $\widetilde{S}_{d_{j}}(s_j)=0$. Hence $\zeta_v(\mathbf{s})=0$.
\end{proof}

Inspired by this result, we introduce the following definition.

\begin{Def}\label{trivialzero}
Let $s_j \in \mathbb{Z}$, $r>1$. $\zeta_v(s_1, \ldots, s_r)=0$ trivially iff either one of the following conditions holds
\begin{itemize}
    \item[(1)] $\exists1 \leq i \leq r$ such that $s_i<0$, $r-i > L_{-s_i} + deg(v)$.
    \item[(2)] $\exists1 \leq i,j\leq r$ such that $s_i,s_j<0$, $deg(v)>r-i > L_{-s_i}$, and $i-j>L_{-s_j}$.
\end{itemize}
We call such zeros \textit{trivial zeros}. Other zeros are called \textit{nontrivial}.

In particular, when $deg(v)=1$, we have $\mathbf{s}$ is a trivial zero iff there exists $1\leq i\leq r$ such that $r-i>L_{-s_i}+1$.
\end{Def}

\begin{Def}
$M=(M_0, \ldots, M_d) \in U_d(k)$ is called the \textit{modest element} if\\ $(M_d, M_{d-1}, \ldots, M_0)$ is lexico-graphically the largest, in more details, $M_d \geq m_d$,\\ $\forall (m_0, \ldots, m_d) \in U_d(k)$, $M_{d-1} \geq m_{d-1}$ for those $(m_0, \ldots, m_d)$ with $m_d=M_d$ and so on.
\end{Def}

\begin{Th}\cite{Shuhui}\label{modest}
Assume $S_d(s) \neq 0$, then the term corresponding to the modest element in $U_d(-s)$ attains the unique minimum degree in $t$ among all summands in $S_d(s)$.
\end{Th}

Similar to the case of MZVs, to study the zeros of $v-$adic MZVs, we need the following result.
\begin{Th}\label{unique}
Let $v=t$, we assume $\widetilde{S}_d(s) \neq 0$, then there is one term attains the unique minimum degree in $t$ among all summands in $\widetilde{S}_d(s)$.
\end{Th}
\begin{proof}
We start with
\begin{align}
\nonumber
\widetilde{S}_d(s) &=\sum_{\substack{\forall i<d,\theta_i \in \mathbb{F}_q\\ \theta_d\in \mathbb{F}_q^*}} (t^d+\theta_{1} t^{d-1}+ \cdots +\theta_d)^{-s} \\\nonumber
&=\sum_{\substack{\forall i<d,\theta_i \in \mathbb{F}_q\\ \theta_d\in \mathbb{F}_q^*}}\sum_{\substack{m_0+\cdots +m_d = -s\\ m_i \geq 0}} \binom{-s}{m_0, \ldots, m_d} \theta_1^{m_1} \cdots \theta_d^{m_d} t^{dm_0+(d-1)m_1+\cdots +m_{d-1}}\\\nonumber
&=(-1)^d\sum_{\substack{m_0+\cdots +m_d = -s\\ m_0 \geq 0, (q-1) | m_i > 0, 1 \leq i < d\\ (q-1)|m_d\geq 0}} \binom{-s}{m_0, \ldots, m_d} t^{dm_0+(d-1)m_1+\cdots +m_{d-1}}\\
&=(-1)^d\sum_{\substack{\oplus_0^d m_i= -s\\ m_0 \geq 0, (q-1) | m_i > 0, 1 \leq i < d\\ (q-1)|m_d\geq 0}} \binom{-s}{m_0, \ldots, m_d} t^{dm_0+(d-1)m_1+\cdots +m_{d-1}} \label{SSd(n)}
\end{align}
where $\oplus_{i=0}^d m_i$ denote sum $\sum_{i=0}^d m_i$ with no carry over of digits base $p$. The third equality comes from exchanging two sum indices and the fact that $\sum_{\theta \in \mathbb{F}_q} \theta^k=-1$ if $k \neq 0, (q-1) | k$ and 0 otherwise. The last equality is the application of Lucas' theorem saying that the multinomial coefficient $\binom{-s}{m_1, \ldots, m_d}$ vanishes in $\mathbb{F}_q$ iff the sum $\sum m_i$ has carry over base $p$. Obviously, $\widetilde{S}_d(s)$ vanishes if sum (\ref{SSd(n)}) is empty. 

We claim that if $dm_0+\cdots+m_{d-1}$ attains the minimum only if $(m_0,\cdots,m_{d-1}+m_d)$ is the most modest element in $S_d(s)$.
Otherwise, we assume $\exists (m_0',\cdots,m_{d-1}')\neq (m_0,\cdots,m_{d-1}+m_d)$ is the most modest element in $U_d(-s)$. then if $m_0'< m_0$, we have $m_0'-m_0>0$ is $q-$even. Hence $(m_0,m_1',\cdots,m_{d-1}'+m_0'-m_0)$ is more modest, contradiction. So we always have $m_0'=m_0$, similar arguments guaranteed that $m_i=m_i'$ for all $i<d-1$, hence, we must have $m_{d-1}'=m_{d-1}+m_d$.

By the \Cref{modest}, we know the uniqueness of the most modest element in $S_{d-1}(s)$ and we let $0<m_{d-1}\leq m_{d-1}+m_d$ to be the maximum $q-$even value such that $m_{d}\geq 0$ is $q-$even and $m_{d-1}\oplus m_d$. Hence, we get that the term with minimum degree in $t$ is unique.
\end{proof}

For $d \geq0$ and $s$ negative, define $\nu_d(s) := v_t(\widetilde{S}_d(s))$, where $v_t$ is the $t-$adic valuation. The above theorem implies
\begin{Cor}
For any $q$ and fixed $s < 0$, we have
$$\nu_{\lfloor L_{-s} \rfloor+1}(s) > \nu_{\lfloor L_{-s} \rfloor}(s) > \cdots > \nu_1(s) \geq \nu_0(s).$$
\end{Cor}
\begin{proof}
Since $\nu_0(s)=v_t(1)=0$ for all $s$, the last inequality is obvious. Assume $1 < d \leq L_{-s}+1$ and let $\mathbf{M}=(M_0, \ldots, M_d)$ corresponding to the unique term in $\widetilde{S}_d(s)$, then \Cref{unique} implies $\nu_d(s)=dM_0+(d-1)M_1+\ldots+M_{d-1}$. Consider $\mathbf{N}=(M_0, \ldots, M_{d-2}, M_{d-1}+M_d)$, $\mathbf{N}\in U_{d-1}(-s)$ and thus $\nu_{d-1}(s) \leq v_t(S_{d-1}(s)) \leq (d-1)M_0+(d-2)M_1+\ldots+M_{d-2} \leq \nu_d(s)$, where the second inequality is equality iff $d=1$ and $M_d=-s$.
\end{proof}

With this corollary, we have the following result.
\begin{Th}
Given $K=\mathbb{F}_q(t)$, $v$ a degree $1$ prime, $\mathbf{s}=(s_1, \cdots, s_r)\in \mathbb{Z}$ with each $s_i \leq 0$, $\zeta_v(\mathbf{s})=0$ if and only if one of the following conditions holds
\begin{itemize}
    \item[(1)]  $r>1$, $\mathbf{s}$ is a trivial zero,
    \item[(2)] $r=1$, $s$ is $q-$even.
\end{itemize}
\end{Th}
\begin{proof}
Without loss of generality, we assume $v=t$.
The case when $r=1$ is done by Goss. So we only need to prove it when $r>1$.
It's equivalent to show that $\zeta_t(\mathbf{s}) \neq 0$ if $\mathbf{s}$ is not a trivial zero. In this case, the sum $\zeta_t(\mathbf{s})=\sum_{d_1 > \cdots > d_r \geq 0} \widetilde{S}_{d_1}(s_1) \cdots \widetilde{S}_{d_r}(s_r)$ is nonempty. In particular, $\widetilde{S}_{r-1}(s_1) \cdots \widetilde{S}_0(s_r) \neq 0$ and 
$$v_t(\widetilde{S}_{r-1}(s_1) \cdots \widetilde{S}_0(s_r))=\sum_{i=1}^r \nu_{r-i}(s_i).$$
For any other term $\widetilde{S}_{d_1}(s_1) \cdots \widetilde{S}_{d_r}(s_r)$ in the sum, $d_i \geq r-i$ for all $i$ and there exist some $j$ such that $d_j > r-j > 0$, thus 
$$v_t(\widetilde{S}_{d_1}(s_1) \cdots \widetilde{S}_{d_r}(s_r))=\sum_{i=1}^r \nu_{d_i}(s_i) > v_t(\widetilde{S}_{r-1}(s_1) \cdots \widetilde{S}_0(s_r)).$$
By strict triangle inequality, $v_t(\zeta_t(\mathbf{s}))=v_t(\widetilde{S}_{r-1}(s_1) \cdots \widetilde{S}_0(s_r))=\sum_{i=1}^r \nu_{r-i}(s_i)$. Applying above theorem, we get $\zeta_t(\mathbf{s})=0$ iff $\exists i$ such that $s_i<0$ and $r-i>L_{-s_i}+1$. Hence, we are done.
\end{proof}

\begin{Def}
Given $K=\mathbb{F}_q(t)$, $v$ a monic prime with degree $d$.
We define
$S_v:=\lim\limits_{\substack{\longleftarrow\\ n}}\mathbb{Z}/(q^d-1)p^n\mathbb{Z}$, and we say $s=c_{-1}+\sum_{i=0}^{\infty}c_i(q^d-1)p^i\in S_v$
 where $q^d-1<c_{-1}\leq 0, p<c_i\leq 0$ is $q-$even iff $q-1|c_{-1}$.
\end{Def}

Thakur has shown in \cite{T04} that all $v-$adic multiple zeta functions are continuous in $S_v^r$, and we have the main theorem. 
\begin{Th}
Given a rational function field $K=\mathbb{F}_q(t)$, a degree $1$ prime $v$, $\mathbf{s}=(s_1, \ldots, s_r) \in S_v^r$, $\zeta_v(\mathbf{s})=0$ if and only if one of the following conditions holds
\begin{itemize}
    \item[(1)] $r>1$, $\mathbf{s}$ is a trivial zero,
    \item[(2)] $r=1$, $s$ is $q-$even. 
\end{itemize}
\end{Th}
\begin{proof}
Without loss of generality, we may assume $v=t$.
If $r=1,s\in \mathbb{Z}$, Goss showed that $\zeta_t(s)=0$ iff $s$ is $q-$even. When $s\notin \mathbb{Z}$ is $q-$odd, by continuation, $\forall s'\in \mathbb{Z}_-$ close enough to $s$ in $S_t$, we have $\mu_1(s)=\mu_1(s')$. And, we have $\mu_1(s')>m_0(s')=0$ by above corollary. Hence, $v_t(\zeta_t(s'))=\mu_0(s')=0$, $\forall s'$ close enough to $s$, i.e., $\zeta_t(s)\neq 0$ when $s$ is not $q-$even in $S_t$.

If $r>1$, we have shown this result when each $s_i\in \mathbb{Z}_{\leq 0}$.
So if $s_i\in S_t-\mathbb{Z}_{\leq 0}$ and $r>1$, we define $\mathbf{s}^n:=(s_1^n,\cdots,s_r^n)$ such that $s_i-s_i^n=\sum_{j=n+1}^{\infty}c_j(q-1)p^j$ where $p< c_j\leq 0$. If $s_i\in \mathbb{Z}_{\leq 0}$, then $s_i^n=s_i$ when $n>>1$. If $s_i\not\in \mathbb{Z}_{\leq 0}$, then we have infinitely many non zero $c_j$'s. Hence, in both cases, we have $\forall d\geq 0, 1\leq i\leq r$, $\nu_t(s_i^n)=\nu_t(s_i)$ when $i>>0$.

Now, apply the argument in last theorem, For any given sequence $d_1>\cdots>d_r\geq 0$, there exists $N\in \mathbb{N}$, such that $\forall n\geq N$, we have
$$v_t(\widetilde{S}_{d_1}(s_1) \cdots \widetilde{S}_{d_r}(s_r))=\sum_{i=1}^r \nu_{d_i}(s_i^n) > v_t(\widetilde{S}_{r-1}(s_1^n) \cdots \widetilde{S}_0(s_r^n))=v_t(\widetilde{S}_{r-1}(s_1) \cdots \widetilde{S}_0(s_r)).$$
By strict triangle inequality, $v_t(\zeta_t(\mathbf{s}))=v_t(\widetilde{S}_{r-1}(s_1) \cdots \widetilde{S}_0(s_r))=\sum_{i=1}^r \nu_{r-i}(s_i^n)$ when $n\geq N$. Applying last theorem, we get $\zeta_t(\mathbf{s})= 0$ iff $\exists i$ such that $s_i\leq 0$ and $r-i>L_{s_i}+1$.
\end{proof}

When $v$ is any prime, we conjecture that the following result holds.
\begin{Conj}
Given $K=\mathbb{F}_q(t)$, $v$ any monic prime, $\mathbf{s}=(s_1, \ldots, s_r) \in S_v^r$, $\zeta_v(\mathbf{s})=0$ if and only if one of the following conditions holds
\begin{itemize}
    \item[(1)] $r>1$, $\mathbf{s}$ is a trivial zero,
    \item[(2)] $r=1$, $s$ is $q-$even. 
\end{itemize}
\end{Conj}

\section{Zeros over all primes}
\begin{Def}
Given a rational function field $K=\mathbb{F}_q(t)$, and a prime $v$, We define $M_v(K):=$ the set consisting of all $v-$adic MZVs $\zeta_v(\mathbf{s})$ at $S_v^r$ over $K$. And we define 
$$M(K):=(M_v(K))_{v \textit{ prime}}=\prod_{v \textit{ prime}}M_v(K),$$
and we say $\mathbf{s}=(s_i)_{i=1}^r\in S_v^r$ is a trivial zero if it's in the intersection of all trivial zeros in $M_v(K)$ for all primes $v$.
\end{Def}

The following proposition shows that to check a specific $\mathbf{s}$ is a trivial zero or not, we only need to check for at most $r$ many primes $v$.
\begin{Prop}\label{P}
Given $\mathbf{s}=(s_i)_{i=1}^r$, $(\zeta_v(\mathbf{s}))_{v}\in M(K)$ are trivial zeros iff $\mathbf{s}$ is a trivial zero for all primes with degree less than or equal to $r$.
\end{Prop}
\begin{proof}
$(\Rightarrow)$
This direction is trivial by the definition.

$(\Leftarrow)$
We only need to show that $\mathbf{s}$ is a trivial zero for all primes $v$ with degree greater than $r$. 
Since $\mathbf{s}$ is a trivial zero for $v$ with degree $r$, in this case, there exists $1 \leq i\leq r$ such that $r-i > L_{-s_i}$ and there exists $j<i$ such that $i-j>L_{-s_j}$. We can see that this condition is independent of the choice of primes $v$ with degree greater than $r$. Hence, we are done.
\end{proof}

Now we claim by giving an example that the set of trivial zeros in $M(K)$ is nonempty.
\begin{Ex}
Given a rational function field $K=\mathbb{F}_q(t)$, $r>1$, and $\mathbf{s}=(s_1,\cdots,s_r)$ where for $1\leq i\leq 3$, $s_i=-p^{n_i}$ for some $n_i\in \mathbb{Z}_{\geq 0}$.

\begin{itemize}
    \item[(1)] If $q>2,r\geq 3$, $\mathbf{s}$ is a trivial zero in $M(K)$.
    \item[(2)] If $q=2,r\geq 5$, $\mathbf{s}$ is a trivial zero in $M(K)$.
\end{itemize}
\end{Ex}
\begin{proof}
By \cref{P}, we only need to check for the primes with degree less than $r+1$.

If $q>2$, we have $L_{-s_i}<1$ for $i\leq 3$.
When $deg(v)=1$, let $i=1$, we have $r-i\geq 2>L_{-s_1}+1$, it's done.
When $r\geq deg(v)>1$, let $i=2,j=1$, we have $deg(v)>r-i=r-2\geq 1>L_{-s_2}$, and $i-j=2-1=1>L_{-s_1}$, it's done.

If $q=2$, we have $L_{-s_i}=1$ for $i\leq 3$.
When $deg(v)=1,2$, let $i=1$, we have $r-i\geq 4>L_{-s_i}+deg(v)$, it's done.
When $r\geq deg(v)>2$, and $i=3,j=1$, we have $deg(v)>r-i=r-3\geq 2>L_{-s_3}$, and $i-j=2>L_{-s_1}$, it's done.
\end{proof}

\section*{acknowledgement}
I am especially thankful to Professor Dinesh S. Thakur for giving lots of valuable suggestions and help reviewing the details of this paper.

\end{document}